\documentstyle[amssymb,amsfonts]{amsart}

\newcommand{\der}{\nabla}

\newenvironment{proof}{\noindent {\bf Proof} }{\endprf\par}
\def \endprf{\hfill  {\vrule height6pt width6pt depth0pt}\medskip}
\def\emph#1{{\it #1}}
\def\textbf#1{{\bf #1}}
\newcommand{\bea}{\begin{eqnarray}}
\newcommand{\eea}{\end{eqnarray}}
\def\beaa{\begin{eqnarray*}}
\def\eeaa{\end{eqnarray*}}

\def\ba{\begin{array}}
\def\ea{\end{array}}
\def\be#1{\begin{equation} \label{#1}}
\def \eeq{\end{equation}}

\def\a{{\alpha}}
\def\b{{\beta}}
\def\ga{\gamma}
\def\de{\delta}
\def\om{\omega}
\def\Om{\Omega}
\def\th{\theta}
\def\ze{\zeta}

\def\thh{\hat{\th}}
\def\lap{\Delta}

\begin{document}
\theoremstyle{plain}
  \newtheorem{theorem}[subsection]{Theorem}
  \newtheorem{conjecture}[subsection]{Conjecture}
  \newtheorem{proposition}[subsection]{Proposition}
  \newtheorem{lemma}[subsection]{Lemma}
  \newtheorem{corollary}[subsection]{Corollary}

\theoremstyle{remark}
  \newtheorem{remark}[subsection]{Remark}
  \newtheorem{remarks}[subsection]{Remarks}

\theoremstyle{definition}
  \newtheorem{definition}[subsection]{Definition}

\include{psfig}
\title[Quasi-Geostrophic Equation with Force]{Global Regularity for the Critical 2-D Dissipative Quasi-Geostrophic Equation with Force}
\author{Sari Ghanem}
\address{Institut de Math\'ematiques de Jussieu, Universit\'e Paris Diderot - Paris VII, 75205 Paris Cedex 13, France}
\email{ ghanem@@math.jussieu.fr}
\maketitle
\begin{abstract}
This is a remark that by using an adaptation of the technique invented by A. Kiselev, F. Nazarov, and A. Voldberg, with a modified scaling argument, we can prove global regularity of the critical 2-D dissipative quasi-geostrophic equation with smooth periodic force, under the assumption that the initial data is smooth and periodic, and the force is $\a$-H\"older continuous in space, $\a$ $>$ 0.
\end{abstract}

\section{Introduction}

The problem of breakdown of solutions of the critical quasi-geostrophic equation with arbitrary smooth initial data was suggested by S. Klainerman in [Kl] as one of the most challenging problems in partial differential equations of the twenty-first century. In an elegant paper, [KNV], A. Kiselev, F. Nazarov and A. Voldberg proved global well-posdness of the critical 2-dimensional dissipative quasi-geostrophic equation with smooth periodic initial data. This note is a remark that by using an adaptation of the technique introduced by Kiselev, Nazarov and Voldberg in [KNV], with a modified scaling argument, we can immediately prove global regularity of the critical 2-dimensional dissipative quasi-geostrophic equation with smooth periodic force, under the assumption that the initial data is smooth and periodic, and the force $\a$-H\"older continuous in space, $\a > 0$.

\subsection{The statement}\

We consider the critical surface quasi-geostrophic equation with force, which we will write as the following:
\bea
\partial_{t}\th(x,t) = u.\der\th(x,t) - (-\lap)^{\frac{1}{2}}\th(x,t) + f(x,t)  \label{1.1}
\eea
where $x \in R^{2}$, $u(x,t) = (-R_{2}\th, R_{1}\th)$, where $R_{1}$ and $R_{2}$ are the usual Riesz transforms in $R^{2}$, $\th(x,t) : R^{2}.R \to R$ is a scalar function, and $f(x,t) : R^{2}.R \to R$ is the force function.

We assume $f$ smooth and periodic on $R^{2}$ (in space), and bounded in space and time, i.e.
\bea
||f(x,t)||_{L^{\infty}} < \infty   \label{assumptionone}
\eea

We also assume $f$ to be $\a$-H\"older continuous with $\a > 0$, i.e. there exist constants $C_{1} \geq 0$ and $\a > 0$ which do not depend on $t$, such that for all $x, y$ in $R^{2}$,
\bea
|f(x,t) - f(y,t)| \leq C_{1}|x - y|^{\a}   \label{assumptiontwo}
\eea

The goal of section \eqref{part I} is to prove the following theorem,

\begin{theorem} \label{th}
Local solutions of the critical surface dissipative quasi-geostrophic equation with smooth periodic force, \eqref{1.1}, with smooth periodic initial data, can be extended to global solutions in time under assumptions \eqref{assumptionone} and \eqref{assumptiontwo} on the force.
\end{theorem}

\begin{remark}
One can prove existence and uniqueness of local solutions of equation \eqref{1.1} under the assumptions of theorem \eqref{th}, by adapting the argument of J. Wu in [Wu]. Thus, theorem \eqref{th} gives global well-posdness for the 2-dimensional critical quasi-geostrophic equation with force on the torus satisfying \eqref{assumptionone} and \eqref{assumptiontwo}.
\end{remark}

\subsection{Strategy of the proof}\ \label{part 0} 

We will prove theorem \eqref{th} by proving that for $\th$ a solution of \eqref{1.1} with smooth periodic initial data $\th_{0}$, $||\der\th||_{L^{\infty}}$ is bounded by a constant depending on $||f||_{L^{\infty}}$, on $C_{1}$ and $\a$ as defined in \eqref{assumptiontwo}, on $||\der\th_{0}||_{L^{\infty}}$, and on the period of $\th_{0}$ and $f$. Once this is achieved, one can show that local solutions can be extended to global solutions in time by adapting the argument shown by A. Kiselev in [K]. To prove such an estimate on $||\der\th||_{L^{\infty}}$ we will use the method of modulus of continuity of A. Kiselev, F. Nazarov, and A. Volberg in [KNV], with a modified scaling argument.

\begin{definition} \label{definitionofmodulusofcontinuity}
We say that a function $\omega$ is a modulus of continuity if $\omega : R_{+} \to R_{+}$ is increasing, continuous, concave, and $\omega(0) = 0$.
\end{definition}

\begin{definition}
We say that $\th$ has modulus of continuity $\omega$, or $\om$ is preserved by $\th$, at time $t$, if for all $x, y \in R^{2}$,
\bea
|\th(x,t) - \th(y,t)| \leq \omega(|x - y|) \label{definitionomegamodulusofcontinuityfortheta}
\eea
\end{definition}

Observe now that if at time $t$, $\th$ has $\omega$ as modulus of continuity, then
\beaa
\frac{|\th(x+h,t) - \th(x,t)|}{|h|} \leq \frac{\omega(|h|)}{|h|}
\eeaa

By taking the limit when $|h| \to 0$ in the above inequality, we obtain for all $x \in R^{2}$
\beaa
|\der\th(x,t)| \leq \omega^{'}(0)
\eeaa

Therefore, by taking the supremum in space in the above inequality, we get that
\bea
||\der\th(x,t)||_{L^{\infty}} \leq \omega^{'}(0)
\eea

Consequently, if we manage to find one special function $\omega$, modulus of continuity, such that given $A$ large enough depending on $||\der\th_{0}||_{L^{\infty}}$, where $\th_{0}$ is the initial data, on $||f||_{L^{\infty}}$, and on the period of $\th_{0}$ and of $f$, such that
\bea
\omega_{A}(\ze) = \omega(A\ze)
\eea
is a modulus of continuity for $\th_{0}$, and $\omega_{A}$ remains preserved for all time $t$ by $\th$, a smooth solution of \eqref{1.1} with $\th_{0}$ as initial data, in the sense of \eqref{definitionomegamodulusofcontinuityfortheta}, then
\bea
||\der\th(x,t)||_{L^{\infty}} \leq A.\omega^{'}(0)  \label{desiredestimate}
\eea

Let's look for such $\omega$:

If
\bea
\omega^{'}(0) = 1  \label{assumptiononomegaderivativeequaltoone}
\eea
 and
\bea
\lim_{\ze \to \infty}  \omega(\ze) = \infty  \label{assumptiononomegagoestoinfinityatinfinity}
\eea
then we notice that since any smooth periodic function $\th_{0}$ is bounded, we can choose $A > 0$ large enough such that $\th_{0}$ has $\omega_{A}(\ze) = \omega(A\ze)$ as modulus of continuity, with $A$ depending on $||\der \th_{0}||_{L^{\infty}}$ and on the period of $\th_{0}$.\

If we also impose on $\om$ to have
\beaa
\lim_{ \ze \to 0^{+}} \omega^{''}(\ze) = -\infty
\eeaa
then, since $\th$ is smooth because $\th_{0}$ and $f$ are smooth, the only way for $\om$ to stop being a modulus of continuity for $\th$ after some time is that there exists a time $T$, and $x, y \in R^{2}$, $x \neq y$, such that
\bea
\th(x,T) - \th(y,T) = \omega_{A}(|x - y|)  \label{limitofomegaceasingtobemodulusofcontinuity}
\eea
and
\bea
\partial_{t}(\th(x,T) - \th(y,T)) \geq 0 \label{thedifferenceincreasingintime}
\eea

Hence, we are going to look for $\omega$ verifying \eqref{assumptiononomegaderivativeequaltoone} and \eqref{assumptiononomegagoestoinfinityatinfinity} such that
\bea
\omega^{''}(0) = - \infty   \label{assumptiononomegasecondderivativeequaltonegativeinfinity}
\eea
and such that at $x, y \in R^{2}$ where \eqref{limitofomegaceasingtobemodulusofcontinuity} is verified, we have
\bea
\partial_{t}(\th(x, T) - \th(y, T)) < 0 \label{differencestrictlydecreasingintime}
\eea

Because of \eqref{thedifferenceincreasingintime}, inequality \eqref{differencestrictlydecreasingintime} will prove that $\omega_{A}$ is preserved by $\th$ for all time $t$, and consequently we will have our estimate.

\textbf{Acknowledgments.} The author would like to thank his advisors, Fr\'ed\'eric H\'elein and Vincent Moncrief, for their continuous advice and support during his PhD studies. The author would also like to thank Alexander Kiselev for suggesting the problem as an exercise, yet the remark presented in this note that consists in modifying the scaling argument in his original technique with F. Nazarov and A. Voldberg, to get the hereby stated result, was not apparent to him, hence the author's interest in posting it. This work was supported by a full tuition fellowship from Universit\'e Paris VII - Institut de Math\'ematiques de Jussieu.\\

\section{Estimate for $||\der\th(x,t)||_{L^{\infty}}$} \label{part I}

Let $\om$ a modulus of continuity, in the sense of \eqref{definitionofmodulusofcontinuity}, such that,
\bea
\lim_{\ze \to \infty} \omega(\ze) &=& \infty  \label{prematureassumptiononomegaone}\\ 
\omega^{'}(0) &=& 1  \label{prematureassumptiononomegatwo}\\
\omega^{''}(0) &=& - \infty \label{thirdassumptiononomega}
\eea

Given an arbitrary smooth periodic initial data $\th_{0}$, since it is a $C^{1}$ function on a compact, we can choose $A$ large enough depending on $||\der\th_{0}||_{L^{\infty}}$ and the period of $\th_{0}$, such that $\th_{0}$ has $\omega_{A}$ as modulus of continuity, i.e. for all $x, y \in R^{2}$, we have
\bea
|\th_{0}(x) - \th_{0}(y)| \leq  \omega_{A}(|x - y|)  \label{inequalityforthechoiceofA}
\eea
This gives for all $x, y \in R^{2}$
\bea
|\th_{0}(\frac{x}{A}) - \th_{0}(\frac{y}{A}) | \leq \omega(|x - y|)
\eea

\begin{definition} \
Let $A$ such that we have \eqref{inequalityforthechoiceofA}, we define
\bea
\hat{\th}(x,t) = \th(\frac{x}{A}, \frac{t}{A})  \label{definitionofthetahat}
\eea
\end{definition}

If $\th(x, t)$ solves \eqref{1.1}, then $\thh(x,t)$ satisfies\

\bea
\partial_{t}\thh(x,t) = \hat{u}.\der\thh(x,t) - (-\lap)^{\frac{1}{2}}\thh(x,t) + \frac{1}{A}f(\frac{x}{A},\frac{t}{A})  \label{fieldequationsatisfiedbyomegahat}
\eea

We would want to find $\om$ preserved by $\thh$ for all time $t$. For this, we will proceed as explained in \eqref{part 0}: 

Let $x, y \in R^{2}$, $x \neq y$, be such that $\thh$ has $\om$ as modulus of continuity for all time $t \leq T$, and
\bea
\thh(x, T) - \thh(y, T) = \omega(|x - y|)  \label{xyTsuchthattheequalityholds}
\eea
Let
\bea
\ze = |x - y|   \label{definitionofze}
\eea

As explained in \eqref{part 0}, we want to find $\om$ such that for $x, y \in R^{2}$ as in \eqref{xyTsuchthattheequalityholds}, we have
\bea
\partial_{t}(\thh(x, T) - \thh(y, T)) < 0 \label{differencethetahatstrictlynegative}
\eea

This will give that $\om$ is preserved by $\thh$ for all time $t$, and consequently $\om_{A}$ is preserved by $\th$ for all time, and therefore we will have our desired estimate \eqref{desiredestimate}.

Computing,
\bea
\notag
&&  \partial_{t}(\thh(x,T) - \thh(y,T)) \\
\notag
&=&  \hat{u}.\der\thh(x,T) - \hat{u}.\der\thh(y,T) - [ (-\lap)^{\frac{1}{2}}\thh(x,T) - (-\lap)^{\frac{1}{2}}\thh(y,T) ]\\
&&+ \frac{1}{A}f(\frac{x}{A},\frac{T}{A} ) - \frac{1}{A}f(\frac{y}{A},\frac{T}{A})   \label{computingthedifferenceoftimederivativesofthetahat}
\eea

\begin{lemma} \label{lemma2.1}
If the function $\thh$ has modulus of continuity $\om$, then $\hat{u} = (-R_{2}\thh, R_{1}\thh)$ has modulus of continuity $\Omega(\ze)$, where
\bea
\Omega(\ze) = B(\int_{0}^{\ze}\frac{\om(\eta)}{\eta}d\eta + \ze\int_{\ze}^{\infty}\frac{\om(\eta)}{\eta^{2}}d\eta) \label{modulusofcontinuityforuhat}
\eea
with some universal constant $B > 0$.
\end{lemma}

\begin{proof}

A sketch of the proof of \eqref{lemma2.1} is in the Appendix of [KNV].

\end{proof}

\begin{lemma} \label{lemma2.2}
For $x, y$ and $T$ as in \eqref{xyTsuchthattheequalityholds}, and $\ze$ defined as in \eqref{definitionofze}, we have
\bea
\hat{u}.\der\thh(x,T) - \hat{u}.\der\thh(y,T) &\leq& \Om(\ze)\om^{'}(\ze)  \geq 0   \label{2.1}\\
\notag
- [ (-\lap)^{\frac{1}{2}}\thh(x,T) - (-\lap)^{\frac{1}{2}}\thh(y,T) ] &\leq& \frac{1}{\pi}\int_{0}^{\frac{\ze}{2}}\frac{\om(\ze + 2\eta) + \om(\ze - 2\eta) -2\om(\ze)}{\eta^{2}}d\eta \\
\notag
&& + \frac{1}{\pi}\int_{\frac{\ze}{2}}^{\infty}\frac{\om(2\eta + \ze) - \om(2\eta - \ze) -2\om(\ze)}{\eta^{2}}d\eta \\ 
&& \label{2.2} \\
\frac{1}{A} [ f(\frac{x}{A},\frac{T}{A}) - f(\frac{y}{A},\frac{T}{A}) ] &\leq&  \frac{C_{1}}{A^{1 + \a}}\ze^{\a}  \label{2.3}
\eea
for some $\a > 0$ and $C_{1} \geq 0$ as in \eqref{assumptiontwo}.
\end{lemma}

\begin{proof}\

To prove \eqref{2.1}, we compute
\bea
\hat{u}.\der\thh(x,T) - \hat{u}.\der\thh(y,T) = \frac{d}{dh}_{|h=0} [ \thh(x + hu(x), T) - \thh(y + hu(y), T) ]  \label{computingthediferenceugradientthetahat}
\eea

We have
\bea
\notag
\thh(x + hu(x), T) - \thh(y + hu(y), T) &\leq& \om(x + hu(x) - y - hu(y)) \\
\notag
&& \text{(because $\om$ is preserved by $\thh$ at time $T$)} \\
&\leq& \om(\ze + h|\hat{u}(x) - \hat{u}(y)|)  \label{omeazeplushtimethedifferenceinu}
\eea
and
\bea
|\hat{u}(x) - \hat{u}(y)| \leq \Om(\ze) \label{boundingthedifferenceinubyOmegaze}
\eea
(by \eqref{lemma2.1})

Since $\om$ is increasing, \eqref{omeazeplushtimethedifferenceinu} and \eqref{boundingthedifferenceinubyOmegaze} give
\bea
\thh(x + hu(x)) - \thh(y + hu(y)) \leq \om(\ze + h\Om(\ze))   \label{inequalityonthediffenceofthetaxplushu}
\eea
\eqref{2.1} comes out after differentiation by injecting \eqref{inequalityonthediffenceofthetaxplushu} in \eqref{computingthediferenceugradientthetahat}.

\eqref{2.2} is proved in [KNV].

\eqref{2.3} comes out from assumption \eqref{assumptiontwo}, that $f$ is $\a$-H\"older continuous, with $\a > 0$:

\beaa
\frac{1}{A} [ f(\frac{x}{A},\frac{T}{A}) - f(\frac{y}{A},\frac{T}{A}) ] &\leq& \frac{C_{1}}{A}\frac{| x - y|^{\a}}{A^{\a}} = \frac{C_{1}}{A^{1 + \a}}\ze^{\a}  
\eeaa

\end{proof}

\subsection{Construction of $\om$}\

Let $\delta > 0$ small enough to be chosen later, $\b = \min\{\frac{1}{2}, \a\}$, where $\a > 0$ is defined as in \eqref{assumptiontwo}, and $0 < \ga \leq \frac{\de}{2}$.

For $0 \leq \ze \leq \delta$, let
\bea
\om(\ze) = \ze - \ze^{1 + \b} \label{definitionofomegaforsmallze}
\eea

For $\ze > \delta $, let
\bea
\om^{'}(\ze) = \frac{\ga}{\ze(4 + log(\frac{\ze}{\delta}))}  \label{definitionderivativeofomegaforbigze}
\eea

\begin{remark} \label{remarkontheconditiontomakeomegamodulusofcontinuity}
For $\de$ small enough, and $0 < \ga \leq \frac{\de}{2}$, $\om$ is a modulus of continuity verifying \eqref{prematureassumptiononomegaone}, \eqref{prematureassumptiononomegatwo}, and \eqref{thirdassumptiononomega}.
\end{remark}

\begin{lemma} \label{lemmathatomegaispreservedbythetahat}
Let $x, y \in R^{2}$ be as in \eqref{xyTsuchthattheequalityholds} with $\om$ as defined in \eqref{definitionofomegaforsmallze} and \eqref{definitionderivativeofomegaforbigze}, and let $\ze = |x - y| > 0$. If we choose $\de$ and $\ga$ small enough, with $0 < \ga \leq \frac{\de}{2}$, then
for all $0 < \ze \leq A.D$, where $D$ is the period of $\th$, we have \eqref{differencethetahatstrictlynegative}, i.e.
\beaa
\partial_{t}(\thh(x, T) - \thh(y, T)) < 0 
\eeaa
\end{lemma}

\begin{proof}

\subsubsection{Checking inequality \eqref{differencethetahatstrictlynegative} for $ 0 < \ze \leq \de$}\

Injecting \eqref{modulusofcontinuityforuhat} in \eqref{2.1}, we get
\bea
\hat{u}.\der\thh(x,T) - \hat{u}.\der\thh(y,T) &\leq& B (\int_{0}^{\ze}\frac{\om(\eta)}{\eta}d\eta + \ze\int_{\ze}^{\infty}\frac{\om(\eta)}{\eta^{2}}d\eta) \om^{'}(\ze)  \label{inequalityonhatuhatthetabymodulusofcontinuity}
\eea

From \eqref{definitionofomegaforsmallze}, we have
\beaa
\om(\eta) \leq \eta
\eeaa
Thus,
\bea
\int_{0}^{\ze}\frac{\om(\eta)}{\eta}d\eta \leq \ze \label{controllingintegralofomegaovereta}
\eea

On the other hand,
\begin{eqnarray*}
\int_{\ze}^{\infty}\frac{\om(\eta)}{\eta^{2}}d\eta &=& \int_{\ze}^{\de}\frac{\om(\eta)}{\eta^{2}}d\eta + \int_{\de}^{\infty}\frac{\om(\eta)}{\eta^{2}}d\eta \\
&=& \int_{\ze}^{\de}(\frac{1}{\eta} - \eta^{-1 + \b})d\eta + [\frac{-\om(\eta)}{\eta}]^{\infty}_{\delta} - \int_{\de}^{\infty}\frac{- \om^{'}(\eta)}{\eta}d\eta \\
&& \text{(by integrating by parts in the second integral)} \\
&=& [\ln \eta]^{\delta}_{\ze} -  [\frac{\eta^{\b}}{\b}]^{\delta}_{\ze} + \frac{\om(\delta)}{\delta} + \ga\int_{\de}^{\infty} \frac{1}{\eta^{2}(4 + \ln(\frac{\eta}{\delta}))}d\eta \\
&\leq&  \ln(\frac{\de}{\ze}) +1 + \ga\int_{\de}^{\infty} \frac{1}{\eta^{2}}d\eta \\
&\leq&  \ln(\frac{\de}{\ze}) +1 +  \frac{\ga}{\de}
\end{eqnarray*}

If we choose $\ga \leq \de$, we obtain
\bea
\int_{\ze}^{\infty}\frac{\om(\eta)}{\eta^{2}}d\eta &\leq& 2 +\ln(\frac{\de}{\ze})  \label{controllingintegralofomegaoveretasquared}
\eea

We also have from \eqref{definitionofomegaforsmallze},
\bea
\om^{'}(\ze) \leq 1 \label{boundonderivativeofomegaforzesmall}
\eea

Injecting \eqref{controllingintegralofomegaovereta}, \eqref{controllingintegralofomegaoveretasquared}, and \eqref{boundonderivativeofomegaforzesmall} in \eqref{inequalityonhatuhatthetabymodulusofcontinuity}, we get
\bea
\notag
\hat{u}.\der\thh(x,T) - \hat{u}.\der\thh(y,T)  &\leq&  B[\ze + \ze(2 + \ln(\frac{\de}{\ze})].1 \\
& \leq&  B[3\ze + \ze \ln(\frac{\de}{\ze})]   \label{finalineqalityonthedifferencehatugradienthattheta}
\eea

On the other hand, \eqref{2.2} has two terms and they are both negative due to the concavity of $\om$. Indeed, the first term in \eqref{2.2} is  $$\frac{1}{\pi}\int_{0}^{\frac{\ze}{2}}\frac{\om(\ze + 2\eta) + \om(\ze - 2\eta) -2\om(\ze)}{\eta^{2}}d\eta$$

If we choose $\de$ small enough, then $\om$ is concave. In addition, $\om^{'''}(\ze) > 0$ due to the choice of $\b$. Hence, using the Taylor series, we can estimate
\beaa
\om(\ze + 2\eta) &\leq& \om(\ze) + 2\om^{'}(\ze)\eta  \\
\om(\ze - 2\eta) &\leq& \om(\ze) -2\om^{'}(\ze)\eta + 2\om^{''}(\ze) \eta^{2}
\eeaa

Therefore, 
\bea
\notag
\frac{1}{\pi}\int_{0}^{\frac{\ze}{2}}\frac{\om(\ze + 2\eta) + \om(\ze - 2\eta) -2\om(\ze)}{\eta^{2}}d\eta &\leq& \frac{1}{\pi}\int_{0}^{\frac{\ze}{2}}\frac{2\om(\ze) + 2\om^{''}(\ze)\eta^{2} -2\om(\ze)}{\eta^{2}}d\eta \\
\notag
&\leq& \frac{1}{\pi}\int_{0}^{\frac{\ze}{2}}2\om^{''}(\ze)d\eta \\
\notag
&\leq& \frac{\ze}{\pi}\om^{''}(\ze)  \\
&\leq& -\frac{\b(1 + \b)}{\pi}\ze^{\b} \label{estimatingthefirsttermintheboundonthedifferencehatugraduenthattheta}
\eea

Whereas to the second term in \eqref{2.2},
\beaa
\frac{1}{\pi}\int_{\frac{\ze}{2}}^{\infty}\frac{\om(2\eta + \ze) - \om(2\eta - \ze) -2\om(\ze)}{\eta^{2}}d\eta  
\eeaa
since $\om$ is concave, we have
\beaa
\om(2\eta + \ze) &=& \om(2\eta - \ze + \ze +\ze) \\
&\leq& \om(2\eta - \ze) + \om(\ze +\ze) \\
&\leq& \om(2\eta - \ze) + 2\om(\ze )
\eeaa

Hence,
\bea
\frac{1}{\pi}\int_{\frac{\ze}{2}}^{\infty}\frac{\om(2\eta + \ze) - \om(2\eta - \ze) -2\om(\ze)}{\eta^{2}}d\eta  &\leq& 0  \label{estimatingthesecondtermintheboundonthedifferencehatugraduenthattheta}
\eea

Injecting \eqref{estimatingthefirsttermintheboundonthedifferencehatugraduenthattheta} and \eqref{estimatingthesecondtermintheboundonthedifferencehatugraduenthattheta} in \eqref{2.2}, we obtain
\bea
- [(-\lap)^{\frac{1}{2}}\thh(x,T) - (-\lap)^{\frac{1}{2}}\thh(y,T) ]  &\leq& -\frac{\b(1 + \b)}{\pi}\ze^{\b}   \label{finalinequalityonthefifferenceofthequareofthelappacianofhattheta}
\eea

Finally, injecting \eqref{2.3}, \eqref{finalineqalityonthedifferencehatugradienthattheta}, and \eqref{finalinequalityonthefifferenceofthequareofthelappacianofhattheta} in \eqref{computingthedifferenceoftimederivativesofthetahat}, we obtain for $0 < \ze \leq \delta$
\beaa
\notag
\partial_{t}(\thh(x,T) - \thh(y,T)) &\leq& B [\ze(3 + \ln(\frac{\de}{\ze}))]  -\frac{\b (1+\b)}{\pi}\ze^{\b} + \frac{C_{1}}{A^{1 + \a}}\ze^{\a} \\
\notag
&\leq& 3B\ze + B\ze \ln(\de) - B\ze \ln(\ze) - \frac{\b(1 + \b)}{\pi}\ze^{\b} + \frac{C_{1}}{A^{1 + \a}}\ze^{\a}
\eeaa

Choosing $\de \leq 1$ and $A \geq 1$, we have
\beaa
\ze^{\a} &\leq& \ze^{\b} \\
A^{1+\b} &\leq& A^{1+\a}
\eeaa

Consequently,
\bea
\notag
\partial_{t}(\thh(x,T) - \thh(y,T))&\leq& B(3\ze + \ze \ln(\de)) - B\ze \ln(\ze) - \ze^{\b}(\frac{\b(1 + \b)}{\pi} - \frac{C_{1}}{A^{1 + \b}})  \\ \label{lastinequalityonthedifferenceoftimederivativesofthetahatforzesmall}
\eea

Choosing $\de$ small enough, and $A$ large enough depending on $C_{1}$ and on $\b$, and therefore on $f$, then \eqref{lastinequalityonthedifferenceoftimederivativesofthetahatforzesmall} would lead to
\beaa
\partial_{t}(\thh(x,T) - \thh(y,T))  &<& 0
\eeaa

\subsubsection{Checking inequality \eqref{differencethetahatstrictlynegative} for $\de \leq \ze \leq A.D$, where $D$ is the period of $\th$}\

From \eqref{computingthedifferenceoftimederivativesofthetahat}, \eqref{modulusofcontinuityforuhat}, \eqref{2.1} and \eqref{2.2}, we have 
\bea
\notag
\partial_{t}(\thh(x,T) - \thh(y,T)) &\leq& \om^{'}(\ze)B[(\int_{0}^{\ze}\frac{\om(\eta)}{\eta}d\eta + \ze\int_{\ze}^{\infty}\frac{\om(\eta)}{\eta^{2}}d\eta)] \\
\notag
&& + \frac{1}{\pi}\int_{0}^{\frac{\ze}{2}}\frac{\om(\ze + 2\eta) + \om(\ze - 2\eta) -2\om(\ze)}{\eta^{2}}d\eta \\
\notag
&& + \frac{1}{\pi}\int_{\frac{\ze}{2}}^{\infty}\frac{\om(2\eta + \ze) - \om(2\eta - \ze) -2\om(\ze)}{\eta^{2}}d\eta \\
&& + \frac{1}{A}(f(\frac{x}{A},\frac{T}{A}) - f(\frac{y}{A},\frac{T}{A}))  \label{anestimateonthetimederivativeofthedifferenceofthetahat}
\eea

We have
\bea
\frac{1}{A}(f(\frac{x}{A},\frac{T}{A}) - f(\frac{y}{A},\frac{T}{A})) \leq \frac{2}{A}||f||_{L^{\infty}}  \label{easyboundonthetermcontainingfsincefisboundedinspaceandtime}
\eea
(from assumption \eqref{assumptionone} on the force).

Whereas to the term
\beaa
\frac{1}{\pi}\int_{0}^{\frac{\ze}{2}}\frac{\om(\ze + 2\eta) + \om(\ze - 2\eta) -2\om(\ze)}{\eta^{2}}d\eta  
\eeaa

since $\om$ is concave, using the Taylor series, we can estimate
\beaa
\om(\ze - 2\eta) &\leq& \om(\ze) - 2\eta \om^{'}(\ze)\\
\om(\ze + 2\eta)  &\leq& \om(\ze) + 2\eta \om^{'}(\ze) \\
\eeaa

Therefore,
\bea
\frac{1}{\pi}\int_{0}^{\frac{\ze}{2}}\frac{\om(\ze + 2\eta) + \om(\ze - 2\eta) -2\om(\ze)}{\eta^{2}}d\eta  \leq 0  \label{non-positivityofoneoftheterms}
\eea

Now, we want to evaluate the term $$\frac{1}{\pi}\int_{\frac{\ze}{2}}^{\infty}\frac{\om(2\eta + \ze) - \om(2\eta - \ze) -2\om(\ze)}{\eta^{2}}d\eta$$
We have
\beaa
\om(2\eta + \ze) &=& \om(2\eta - \ze + 2\ze) \\
&\leq& \om(2\eta - \ze) + \om(2\ze)
\eeaa
(by concavity).

Hence,
\bea
\frac{1}{\pi}\int_{\frac{\ze}{2}}^{\infty}\frac{\om(2\eta + \ze) - \om(2\eta - \ze) -2\om(\ze)}{\eta^{2}}d\eta  \leq \frac{1}{\pi}\int_{\frac{\ze}{2}}^{\infty}\frac{\om(2\ze) - 2\om(\ze)}{\eta^{2}}d\eta \label{inequalityonthesecondtermafterusingconcavity}
\eea

Since $\om$ is concave, we also have
\bea
\notag
\om(2\ze) &\leq& \om(\ze) + \ze\om^{'}(\ze) \\
\notag
&\leq& \om(\ze) + \frac{\ga}{4 + \ln(\frac{\ze}{\de})} \\
&\leq& \om(\ze) + \frac{\ga}{4}  \label{inequalityofomegatwoze}
\eea

If we choose $\ga < \frac{\de}{2}$, \eqref{inequalityofomegatwoze} will lead to 
\bea
\om(2\ze)  &\leq& \om(\ze) + \frac{\de}{8} \label{inequalityonomegatwozeafterchoosinggasmall}
\eea

If we choose $\de$ small enough, we will have
\bea
\de^{1+\b} \leq \frac{\de}{2}  \label{smallnessondecondition}
\eea
then, from \eqref{definitionofomegaforsmallze} and \eqref{definitionderivativeofomegaforbigze} we will get
\bea
\frac{\de}{2} \leq \om(\de) \leq \om(\ze) \label{inequalityondeovertwo}
\eea
Injecting \eqref{inequalityondeovertwo} in \eqref{inequalityonomegatwozeafterchoosinggasmall}, we obtain 
\beaa
\om(2\ze) \leq \frac{3}{2}\om(\ze)   \label{verylastinequalityonomega2ze}
\eeaa

Consequently,
\bea
\frac{1}{\pi}\int_{\frac{\ze}{2}}^{\infty}\frac{\om(2\eta + \ze) - \om(2\eta - \ze) -2\om(\ze)}{\eta^{2}}d\eta    \leq -\frac{1}{2\pi}\int_{\frac{\ze}{2}}^{\infty}\frac{\om(\ze)}{\eta^{2}}d\eta =  -\frac{1}{\pi}\frac{\om(\ze)}{\ze}  \label{estimateoftheothernon-positiveterms}
\eea
(from \eqref{inequalityonthesecondtermafterusingconcavity}).

Now, we would want to evaluate the term $$(\int_{0}^{\ze}\frac{\om(\eta)}{\eta}d\eta + \ze\int_{\ze}^{\infty}\frac{\om(\eta)}{\eta^{2}}d\eta)$$

We have,
\begin{eqnarray*}
\int_{0}^{\ze}\frac{\om(\eta)}{\eta}d\eta &\leq& \int_{0}^{\de}\frac{\om(\eta)}{\eta}d\eta + \int_{\de}^{\ze}\frac{\om(\eta)}{\eta}d\eta  \\
&\leq& \de + \om(\ze)\ln(\frac{\ze}{\de}) 
\end{eqnarray*}

If we choose $\de$ small enough as before in \eqref{smallnessondecondition}, so that $\om(\ze) \geq \om(\de) \geq \frac{\de}{2}$, we obtain
\bea
\notag
\int_{0}^{\ze}\frac{\om(\eta)}{\eta}d\eta  &\leq& 2\om(\ze) + \om(\ze)\ln(\frac{\ze}{\de}) \\
&\leq&  \om(\ze)( 2 + \ln(\frac{\ze}{\de}))  \label{estimateonintegralofomegaetaovereta}
\eea

On the other hand, integrating by parts and using \eqref{definitionderivativeofomegaforbigze}, we can evaluate
\beaa
\int_{\ze}^{\infty}\frac{\om(\eta)}{\eta^{2}}d\eta &=& \frac{\om(\ze)}{\ze} + \ga\int_{\ze}^{\infty} \frac{1}{\eta^{2}(4 + \ln(\frac{\eta}{\delta}))}d\eta \\
&\leq& \frac{\om(\ze)}{\ze} + \frac{\ga}{\ze} 
\eeaa

Consequently, if we choose $\ga \leq \frac{\de}{2}$, with $\de$ small enough as in \eqref{smallnessondecondition}, then from \eqref{inequalityondeovertwo} we get
\bea
\int_{\ze}^{\infty}\frac{\om(\eta)}{\eta^{2}}d\eta &\leq&  2\frac{\om(\ze)}{\ze} \label{inequalityonintegralomegaetaoveretasquared}
\eea

Hence, from \eqref{estimateonintegralofomegaetaovereta} and \eqref{inequalityonintegralomegaetaoveretasquared}, we get
\bea
\notag
(\int_{0}^{\ze}\frac{\om(\eta)}{\eta}d\eta + \ze\int_{\ze}^{\infty}\frac{\om(\eta)}{\eta^{2}}d\eta) &\leq&   \om(\ze)( 2 + \ln(\frac{\ze}{\de}))  + 2\om(\ze)  \\
&\leq& \om(\ze)( 4 + \ln(\frac{\ze}{\de}))      \label{anotherestimateonthemodulusofcontinuityofuhatgradientthetahat}
\eea

Finally, injecting \eqref{easyboundonthetermcontainingfsincefisboundedinspaceandtime}, \eqref{non-positivityofoneoftheterms}, \eqref{estimateoftheothernon-positiveterms}, and \eqref{anotherestimateonthemodulusofcontinuityofuhatgradientthetahat} in \eqref{anestimateonthetimederivativeofthedifferenceofthetahat}, we obtain
\begin{eqnarray*}
\partial_{t}(\thh(x,T) - \thh(y,T)) &\leq&  B\om(\ze)(4 + \ln(\frac{\ze}{\de}) )\om^{'}(\ze) - \frac{1}{\pi}\frac{\om(\ze)}{\ze} + 2\frac{||f||_{L^{\infty}}}{A} 
\end{eqnarray*}

Therefore, from \eqref{definitionderivativeofomegaforbigze} we have
\begin{eqnarray*}
\partial_{t}(\thh(x,T) - \thh(y,T)) &\leq& B\ga\frac{\om(\ze)}{\ze} - \frac{1}{\pi}\frac{\om(\ze)}{\ze} + 2\frac{||f||_{L^{\infty}}}{A} \\
&\leq& \frac{\om(\ze)}{\ze}(  B\ga - \frac{1}{\pi}) + 2\frac{||f||_{L^{\infty}}}{A} 
\end{eqnarray*}

If we choose $\ga$ small enough, we get $$B\ga - \frac{1}{\pi} < 0$$
then, we get for all $\de \leq \ze \leq A.D$, where $D$ is the period of $\th$,
\beaa
\partial_{t}(\thh(x,T) - \thh(y,T)) \leq \frac{\om(A.D)}{A.D}(  B\ga - \frac{1}{\pi}) + 2\frac{||f||_{L^{\infty}}}{A} 
\eeaa

Since $\om$ is increasing, we can choose $A$ large enough depending on $D$ and $||f||_{L^{\infty}}$, such that $$\frac{\om(A.D)}{A.D}(  B\ga - \frac{1}{\pi}) + 2\frac{||f||_{L^{\infty}}}{A} < 0 $$

\end{proof}

Remark \eqref{remarkontheconditiontomakeomegamodulusofcontinuity} and lemma \eqref{lemmathatomegaispreservedbythetahat} show that for $\de$ and $\ga$ chosen small enough, with $0 < \ga \leq \frac{\de}{2}$, $\om$ is preserved by $\thh$ for $0 \leq \ze \leq A.D$, where $D$ is the period of $\th$, for all time $t$. Since $\th$ is periodic of period $D$ depending on the period of $\th_{0}$ and of $f$, then $\thh$ is periodic of period $A.D$, and since $\om$ is increasing, we have by then that for all $\ze \geq A.D$ and for all time $t$,
\beaa
\thh(x, t) - \thh(y, t)  \leq \om(\ze)
\eeaa
Therefore, $\om_{A}$ is preserved by $\th$ for all time. Consequently, from \eqref{desiredestimate} we have
\beaa
||\der\th||_{L^{\infty}} \leq A\om^{'}(0) \leq A
\eeaa
where $A$ depends only on $||f||_{L^{\infty}}$, on $C_{1}$ and $\b = \min\{\frac{1}{2}, \a\}$, on $||\der\th_{0}||_{L^{\infty}}$, on the period of $\th_{0}$, and on the period $D$ of $\th$ (which is given by the period of $\th_{0}$ and the period of $f$). If $A$ is finite, this gives that local solutions of \eqref{1.1} can be extended globally in time.\\

\end{document}